\documentclass[a4paper,11pt]{article}
\usepackage[nointlimits]{amsmath}
\usepackage{amssymb,amsmath,amsthm}
\usepackage{amsfonts}
\usepackage{latexsym}
\usepackage{graphicx}
\usepackage{color}
\usepackage{layout}
\usepackage[english]{babel}
\numberwithin{equation}{section} 

\def\sign{\mbox{sign}}
\def\ZZ{\mathbb Z}

\newtheorem{thm}{Theorem}[section]
\newtheorem{lem}[thm]{Lemma}
\newtheorem{cor}[thm]{Corollary}

\begin{document}

\title{\bf Congruences involving \\ alternating multiple harmonic sum}

\author{{\sc Roberto Tauraso}\\
Dipartimento di Matematica\\
Universit\`a di Roma ``Tor Vergata'', Italy\\
{\tt tauraso@mat.uniroma2.it}\\
{\tt http://www.mat.uniroma2.it/$\sim$tauraso}
}

\date{}
\maketitle
\begin{abstract}
\noindent We show that for any prime prime $p\not=2$
$$\sum_{k=1}^{p-1} {(-1)^k\over k}{-{1\over 2} \choose k} \equiv -\sum_{k=1}^{(p-1)/2}{1\over k} \pmod{p^3}$$
by expressing the l.h.s. as a combination of alternating multiple harmonic sums.
\end{abstract}

\section{Introduction}
In \cite{VHa:96} Van Hamme presented several results and conjectures
concerning a curious analogy between the values of certain hypergeometric series
and the congruences of some of their partial sums modulo power of prime.
In this paper we would like to discuss a new example of this analogy. Let us consider
\begin{eqnarray*}
\sum_{k=1}^{\infty}{(-1)^k\over k}{-{1\over 2} \choose k}&=&
\left(1\over 2\right)+{1\over 2}\left(1\cdot 3\over 2\cdot 4\right)
+{1\over 3}\left(1\cdot 3\cdot 5\over 2\cdot 4 \cdot 6\right)
+{1\over 4}\left(1\cdot 3\cdot 5 \cdot 7\over 2\cdot 4 \cdot 6\cdot 8\right)
+\cdots\\
&=&\int_0^{-1}{1\over x}\left({1\over \sqrt{1+x}}-1\right)\,dx=
-2\left[\log\left({1+\sqrt{1+x}\over 2}\right)\right]_0^{-1}
=2\log 2.
\end{eqnarray*}
Let $p$ be a prime number, what's the $p$-adic analogue of the above result?

\noindent The real case suggests to replace the logarithm
with some $p$-adic function which behaves in a similar way.
It turns out that the right choice is the {\sl Fermat quotient} 
$$q_p(x)={x^{p-1}-1 \over p}$$
(which is fine since $q_p(x \cdot y)\equiv q_p(x)+q_p(y)$ (mod $p$)),
and, as shown in \cite{SuzwTa:09}, the following congruence holds for any prime $p\not=2$
$$\sum_{k=1}^{p-1}{(-1)^k\over k}{-{1\over 2} \choose k}\equiv 2\,q_p(2) \pmod{p}.$$

Here we improve this result to the following statement.

\begin{thm}\label{T11} For any prime $p>3$
\begin{eqnarray*}
\sum_{k=1}^{p-1} {(-1)^k\over k}{-{1\over 2} \choose k} 
&\equiv& 2q_p(2)-pq_p(2)^2+{2\over 3}p^2q_p(2)^3+{7\over 12}p^2 B_{p-3} \\
&\equiv& -\sum_{k=1}^{(p-1)/2}{1\over k} \pmod{p^3}
\end{eqnarray*}
where $B_n$ is the $n$-th Bernoulli number.
\end{thm}

In the proof we will employ some new congruences for
alternating multiple harmonic sums which are interesting in themselves such as
\begin{align*}
&H(-1,-2;p-1)=\sum_{0<i<j<p}{(-1)^{i+j}\over ij^2}\equiv -{3\over 4}B_{p-3}           &\pmod{p}\,,\\
&H(-1,-1,1;p-1)=\sum_{0<i<j<k<p}{(-1)^{i+j}\over ijk}\equiv q_p(2)^3+{7\over 8}B_{p-3}&\pmod{p}.
\end{align*}

\section{Alternating multiple harmonic sums}
Let $r>0$ and let $(a_1,a_2,\dots,a_r)\in (\ZZ^*)^r$.
For any $n\geq r$, we define the {\it alternating multiple harmonic sum} as
$$H(a_1,a_2,\dots,a_r;n)=
\sum_{1\leq k_1<k_2<\dots<k_r\leq n}\; \prod_{i=1}^r{\sign(a_i)^{k_i}\over k_i^{|a_i|}}.$$
The integers $r$ and $\sum_{i=1}^r |a_i|$ are respectively the {\it depth} and the {\it weight} of the harmonic sum.
 
From the definition one derives easily the {\it shuffle relations}:
\begin{align*}
&H(a;n)\cdot H(b;n)=H(a,b;n)+H(b,a;n)+H(a\oplus b;n)\\
&H(a,b;n)\cdot H(c;n)=H(c,a,b;n)+H(a,c,b;n)+H(a,b,c;n)\\
&\hspace{40mm}+H(a\oplus b,c;n)+H(a,b\oplus c;n)
\end{align*}
where $a\oplus b=\sign(ab)(|a|+|b|)$.

\noindent Moreover, if $p$ is a prime, by replacing $k_i$ with $p-k_i$ we get the
{\it reversal relations}:
\begin{align*}
&H(a,b;p-1)\equiv H(b,a;p-1)(-1)^{a+b}\sign(ab) &\pmod{p}\,,\\
&H(a,b,c;p-1)\equiv H(c,b,a;p-1)(-1)^{a+b+c}\sign(abc) &\pmod{p}.
\end{align*}

The values of several {\it non-alternating} (i. e. when all the indices are positive)
harmonic sums modulo a power of prime are well known:

\begin{enumerate} 

\item[(i).] (\cite{Ho:07}, \cite{ZC:07}) for $a,r>0$ and for any prime $p>ar+2$
$$
H(\left\{a\right\}^r;p-1)\equiv \left\{
\begin{array}{lll}
(-1)^{r}{a(ar+1)\over 2(ar+2)}\,p^2\,B_{p-ar-2}  &\pmod{p^3} &\mbox{if $ar$ is odd}\\ \\
(-1)^{r-1}{a\over ar+1}p\,B_{p-ar-1}  &\pmod{p^2} &\mbox{if $ar$ is even}
\end{array}
\right.;
$$

\item[(ii).] (\cite{Sunzh:00}) for any prime $p>3$
$$
H\left(1;{p-1\over 2}\right)\equiv
-2q_p(2)+p q_p(2)^2-{2\over 3}\,p^2 q_p(2)^3-{7\over 12}\,p^2\,B_{p-3}  \pmod{p^3}.
$$
and for $a>1$ and for any prime $p>a+1$
$$H\left(a;{p-1\over 2}\right)\equiv\left\{
\begin{array}{lll}
-{2^a-2\over a}\,B_{p-a} &\pmod{p} &\mbox{if $a$ is odd}\\ \\
{a(2^{a+1}-1)\over 2(a+1)}\,p\,B_{p-a-1} &\pmod{p^2} &\mbox{if $a$ is even}
\end{array}
\right.;$$

\item[(iii).] (\cite{Ho:07}, \cite{Zh:06}) for $a,b>0$ and for any prime $p>a+b+1$
$$H(a,b;p-1)\equiv {(-1)^b\over a+b}{a+b\choose a} \,B_{p-a-b} \pmod{p}$$
(note that $B_{2n+1}=0$ for $n>0$).

\end{enumerate}

The following result will allow us to compute the mod $p$ values 
of multiple harmonic sums of depth $\leq 2$ when the indices are all negative.

\begin{thm}\label{T21} Let $a,b>0$ then for any prime $p\not=2$
\begin{align*}
&H(-a;p-1)=-H(a;p-1)+{1\over 2^{a-1}} H\left(a;{p-1\over 2}\right),\\
&2H(-a,-a;p-1)=H(-a;p-1)^2-H(2a;p-1),
\end{align*}
and
$$H(-a,-b;p-1)\equiv
-\left(1-{1\over 2^{a+b-1}}\right)\,H(a,b;p-1)-{(-1)^b\over 2^{a+b-1}} H\left(a;{p-1\over 2}\right)H\left(b;{p-1\over 2}\right)
\pmod{p}.
$$
\end{thm}
\begin{proof} The shuffling relation given by $H(-a;p-1)^2$ yields the second equation.
As regards the first equation we simply observe that $(-1)^i/i^a$  is positive if and only if $i$ is even.
We use a similar argument for the congruence: since  $(-1)^{i+j}/(i^a j^b)$ is positive if and only
if $i$ and $j$ are both even or if $(p-i)$ and $(p-j)$ are both even then
$$H(-a,-b;p-1) \equiv -H(a,b;p-1)+
{2\over 2^{a+b}}\left(
H\left(a,b;{p-1\over 2}\right)+(-1)^{a+b}H\left(b,a;{p-1\over 2}\right)
\right).$$
Moreover, by decomposing the sum $H(a,b;p-1)$ we obtain
$$
H(a,b;p-1)\equiv H\left(a,b;{p-1\over 2}\right)+
H\left(a;{p-1\over 2}\right)(-1)^b H\left(b;{p-1\over 2}\right)
+(-1)^{a+b}H\left(b,a;{p-1\over 2}\right).
$$
that is 
$$H\left(a,b;{p-1\over 2}\right)+(-1)^{a+b}H\left(b,a;{p-1\over 2}\right)
\equiv H(a,b;p-1)-H\left(a;{p-1\over 2}\right)(-1)^b H\left(b;{p-1\over 2}\right).$$
and the congruence follows immediately.
\end{proof}

\begin{cor}\label{C22} For any prime $p>3$
\begin{align*}
&H(-1;p-1)\equiv -2q_p(2)+pq_p(2)^2-{2\over 3}p^2 q_p(2)^3-{1\over 4}p^2\,B_{p-3} &\pmod{p^3}\,, \\
&H(-1,-1;p-1)\equiv 2q_p(2)^2-2pq_p(2)^3-{1\over 3}p\,B_{p-3}&\pmod{p^2}.
\end{align*}
Moreover for $a>1$ and for any prime $p>a+1$
$$H(-a;p-1)\equiv -{2^a-2\over a2^{a-1}}B_{p-a} \pmod{p}.$$
\end{cor}
\begin{proof} The proof is straightforward: apply Theorem \ref{T21}, (i), (ii), and (iii). 
\end{proof}

The following theorem is a variation of a result presented in \cite{ZhSuzw:09}.

\begin{thm}\label{T23} Let $r>0$ then for any prime $p>r+1$
$$H(\{1\}^{r-1},-1;p-1)\equiv (-1)^{r-1}\sum_{k=1}^{p-1}{2^k\over k^r} \pmod{p}.$$
\end{thm}

\begin{proof} For $r\geq 1$, let
$$F_r(x)=\sum_{0<k_1<\dots<k_r<p} {x^{k_r}\over k_1\cdots k_r} \in \ZZ_p[x]
\quad\mbox{and}\quad f_r(x)=\sum_{0<k<p} {x^k\over k^r}\in \ZZ_p[x].$$ 
We show by induction that
$$F_r(x)\equiv (-1)^{r-1} f_r(1-x) \pmod{p}$$
then our congruence follows by taking $x=-1$.

\noindent For $r=1$, since ${p\choose k}=(-1)^{k-1}{p\over k}\pmod{p^2}$ for $0<k<p$ then
$$f_1(x)\equiv {1\over p}\sum_{k=1}^{p-1}(-1)^{k-1}{p\choose k}{x^k}
=-{1\over p}\sum_{k=1}^{p-1}{p\choose k}{(-x)^k}=
{1-(1-x)^p-x^p\over p}\pmod{p}.$$
Hence $F_1(x)=f_1(x)\equiv f_1(1-x)\pmod{p}$. 

\noindent Assume that $r>1$, then the formal derivative yields
\begin{eqnarray*}
{d\over dx} F_r(x)&=&\sum_{0<k_1<\dots<k_r<p} {k_rx^{k_r-1}\over k_1\cdots k_r}
=\sum_{0<k_1<\dots<k_{r-1}<p} {1\over k_1\cdots k_{r-1}}\sum_{k_r=k_{r-1}+1}^{p-1} x^{k_r-1}\\
&=&\sum_{0<k_1<\dots<k_{r-1}<p} {1\over k_1\cdots k_{r-1}}\cdot {x^{p-1}-x^{k_{r-1}}\over x-1}\\
&=&{x^{p-1}\over x-1}\, H(\{1\}^{r-1};p-1)-{1\over x-1}\, F_{r-1}(x)
\equiv {F_{r-1}(x)\over 1-x} \pmod{p}.
\end{eqnarray*}
Moreover
$${d\over dx} f_r(1-x)=-\sum_{0<k<p} {(1-x)^{k-1}\over k^{r-1}}=-{f_{r-1}(1-x)\over 1-x}$$
Hence, by the induction hypothesis
$$
(1-x){d\over dx} \left(F_r(x)+(-1)^r f_r(1-x)\right)\equiv
F_{r-1}(x)+(-1)^{r-1} f_{r-1}(1-x)\equiv 0 \pmod{p}.
$$
Thus $F_r(x)+(-1)^r f_r(1-x)\equiv c_1$ (mod $p$) for some constant $c_1$ since this polynomial
has degree $<p$. Substituting in $x=0$ we find that by (i)
$$F_r(x)+(-1)^r f_r(1-x)\equiv c_1\equiv F_r(0)+(-1)^r f_r(1)=(-1)^r H(r;p-1)\equiv 0 \pmod{p}.$$
\end{proof}

With the next two corollaries we have a complete the list of the mod $p$ values of 
the alternating multiple harmonic sums of depth and weight $\leq 3$.

\begin{cor}\label{C24} The following congruences mod $p$ hold for any prime $p>3$
\begin{align*}
&H(1,-1;p-1)\equiv -H(-1,1;p-1)\equiv q_p(2)^2\,, \\
&H(-1,2;p-1)\equiv H(1,-2;p-1)\equiv H(2,-1;p-1)\equiv H(-2,1;p-1)\equiv {1\over 4} B_{p-3}\,,\\
&H(-1,-2;p-1)\equiv -H(-2,-1;p-1)\equiv -{3\over 4}B_{p-3}\,.
\end{align*}
\end{cor}

\begin{proof} By Theorem \ref{T23} and  by \cite{Gr:04} 
$$H(1,-1;p-1)\equiv-\sum_{k=1}^{p-1}{2^k\over k^2}\equiv q_p(2)^2 \pmod{p}.$$
By (i) and by the shuffling relation given by the product $H(-1;p-1)H(2;p-1)$ we get
$$H(-1,2;p-1)={1\over 2}H(-1;p-1)H(2;p-1)-{1\over 2}\; H(-3;p-1)\equiv {1\over 4} B_{p-3} \pmod{p}.$$
By (ii) and by Theorem \ref{T21}
$$H(-1,-2;p-1)\equiv-{3\over 4}H(1,2;p-1)-{1\over 4} H\left(1;{p-1\over 2}\right)H\left(2;{p-1\over 2}\right)
\equiv -{3\over 4}B_{p-3}
\pmod{p}.$$
The remaining congruences follow by applying the reversal relation of depth $2$.
\end{proof}

\begin{cor}\label{C25} The following congruences mod $p$ hold for any prime $p>3$
\begin{align*}
&H(-1,1,-1;p-1)\equiv 0 \,,\\
&H(1,1,-1;p-1)\equiv H(-1,1,1;p-1)\equiv -{1\over 3}q_p(2)^3-{7\over 24}B_{p-3}\,, \\
&H(-1,-1,1;p-1)\equiv -H(1,-1,-1;p-1)\equiv q_p(2)^3+{7\over 8}B_{p-3}\,, \\
&H(1,-1,1;p-1)\equiv {2\over 3}q_p(2)^3+{1\over 12}B_{p-3}\,,\\
&H(-1,-1,-1;p-1)\equiv -{4\over 3}q_p(2)^3-{1\over 6}B_{p-3}\,. 
\end{align*}
\end{cor}

\begin{proof} 
By the reversal relation of depth 3,
$H(-1,1,-1;p-1)\equiv -H(-1,1,-1;p-1)\equiv 0$.
By Theorem \ref{T23} and by \cite{DiSk:06}
$$H(1,1,-1;p-1)\equiv\sum_{k=1}^{p-1}{2^k\over k^3}\equiv -{1\over 3}q_p(2)^3+{7\over 12}H(-3,p-1)
\equiv -{1\over 3}q_p(2)^3-{7\over 24}B_{p-3}\pmod{p}.$$
By the shuffling relations given by the products
$$H(1,-1;p-1)H(-1;p-1),\;H(1,-1;p-1)H(1;p-1),\,\mbox{and}\;H(-1,-1;p-1)H(-1;p-1)$$
we respectively find that
\begin{align*}
&2H(1,-1,-1;p-1)\equiv H(1,-1;p-1)H(-1;p-1)-H(1,2;p-1)-H(-2,-1;p-1)\,, \\
&H(1,-1,1;p-1)\equiv  -2H(1,1,-1,p-1)-2H(2,-1;p-1)\,,\\
&3H(-1,-1,-1;p-1)\equiv H(-1,-1;p-1)H(-1;p-1)-2H(2,-1;p-1).
\end{align*}
The remaining congruences follow by applying the reversal relation of depth 3.
\end{proof}

\section{Proof of Theorem \ref{T11}}\label{sec:two}

The following useful identity appears in \cite{SuzwTa:09}. Here 
we give an alternate proof by using Riordan's array method 
(see \cite{Sp:94} for more examples of this technique).

\begin{thm}\label{T31} Let  $n\geq d>0$
$$ d\sum_{k=1}^{n} {2k \choose k+d}\,{x^{n-k}\over k}
=\sum_{k=0}^{n-d} {2n \choose n+d+k} v_k-{2n \choose n+d}$$
where $v_0=2$, $v_1=x-2$ and $v_{k+1}=(x-2)v_k-v_{k-1}$ for $k\geq 1$.
\end{thm}

\begin{proof} We first note that
\begin{eqnarray*}
{2k \choose k+d}&=&
{2k \choose k-d}=(-1)^{k-d}{-k-d-1 \choose k-d}\\
&=& [z^{k-d}]{1\over (1-z)^{k+d+1}}=[z^{-1}]
{z^{d-1}\over(1-z)^{d+1}}\cdot \left({1 \over z(1-z)}\right)^k.
\end{eqnarray*}
Since the residue of a derivative is zero then  
\begin{eqnarray*}
d\sum_{k=1}^{n} {2k \choose k+d}\,{x^{n-k}\over k}
&=&[z^{-1}]\, x^n\, {dz^{d-1}\over (1-z)^{d+1}}\, G\left({1\over x z(1-z)}\right)\\
&=&-[z^{-1}]\, x^n\, {z^{d}\over (1-z)^{d}}\, G'\left({1\over x z(1-z)}\right)\cdot \left({1\over x z(1-z)}\right)'\\
&=&[z^{-1}]\, {z^{d-n-1}\over (1-z)^{n+d+1}}\, {1-x^nz^n(1-z)^n\over 1-xz+xz^2}\cdot (1-2z)\\
&=&[z^{-1}]\, {z^{d-n-1}\over (1-z)^{n+d+1}}\, {1-2z\over 1-xz+xz^2}.
\end{eqnarray*}
where $G(z)=\sum_{k=1}^n {z^k\over k}$ and $G'(z)=\sum_{k=1}^n z^{k-1}={1-z^n\over 1-z}$. 
Moreover
\begin{eqnarray*}
{2n\choose n+d+k}&=&{2n \choose n-d-k}=(-1)^{n-d-k} {-n-d-k-1\choose n-d-k}\\
&=&[z^{n-d-k}]{1\over (1-z)^{n+d+k+1}}=
[z^{-1}]{z^{d-n-1}\over (1-z)^{n+d+1}}\cdot \left({z\over 1-z}\right)^k
\end{eqnarray*}
Letting $F(z)=\sum_{k=0}^{\infty} v_k z^k={2-(x-2)z\over 1-(x-2)z+z^2}$ then
\begin{eqnarray*}
\sum_{k=0}^{n-d} {2n \choose n+d+k} v_k-{2n \choose n+d}
&=&
[z^{-1}]{z^{d-n-1}\over (1-z)^{n+d+1}}\cdot F\left({z\over 1-z}\right)
-[z^{-1}]{z^{d-n-1}\over (1-z)^{n+d+1}}\\
&=&[z^{-1}]\, {z^{d-n-1}\over (1-z)^{n+d+1}}\,\left({(2-xz)(1-z)\over 1-xz+xz^2}-1\right)\\
&=&[z^{-1}]\, {z^{d-n-1}\over (1-z)^{n+d+1}}\, {1-2z\over 1-xz+xz^2}
.
\end{eqnarray*}
\end{proof}

\begin{cor}\label{C32} For any $n>0$
$$4^n\sum_{k=1}^{n} {-{1\over 2} \choose k} \,{(-1)^k\over k}=
-4(-1)^{n}\sum_{d=0}^{n-1}{(-1)^{d}\over n-d}\sum_{j=0}^{d-1} {2n \choose j}-2(-1)^{n}\sum_{d=0}^{n-1}{(-1)^{d}\over n-d}{2n \choose d}.
$$
\end{cor}

\begin{proof} Since
$$0=\sum_{d=-k}^{k}(-1)^{d}{2k\choose k+d}
={2k\choose k}+2\sum_{d=1}^{k} (-1)^d {2k\choose k+d}$$
then for any $n\geq k$
$$(-1)^k{-{1\over 2}\choose k}=4^{-k}{2k\choose k}=-2\cdot4^{-k}\sum_{d=1}^{n} (-1)^d {2k\choose k+d}.$$
For $x=4$ then $v_k=2$ for all $k\geq 0$ and by Theorem \ref{T31}
\begin{eqnarray*}
4^n\sum_{k=1}^{n} {(-1)^k\over k}{-{1\over 2} \choose k}&=&
-2\sum_{k=1}^{n}{4^{n-k}\over k}\sum_{d=1}^{n} (-1)^d {2k\choose k+d}
=-2\sum_{d=1}^{n}(-1)^d \sum_{k=1}^{n} {4^{n-k}\over k}{2k\choose k+d}\\
&=&-4\sum_{d=1}^{n}{(-1)^d\over d}\sum_{k=0}^{n-d} {2n \choose n+d+k}+2\sum_{d=1}^{n}{(-1)^d\over d}{2n \choose n+d}\\
&=&-4\sum_{d=1}^{n}{(-1)^d\over d}\sum_{k=1}^{n-d} {2n \choose n-d-k}-2\sum_{d=1}^{n}{(-1)^d\over d}{2n \choose n-d}\\
&=&-4(-1)^{n}\sum_{d=0}^{n-1}{(-1)^{d}\over n-d}\sum_{j=0}^{d-1} {2n \choose j}-2(-1)^{n}\sum_{d=0}^{n-1}{(-1)^{d}\over n-d}{2n \choose d}.
\end{eqnarray*}
\end{proof}

We will make use of the following lemma.

\begin{lem}\label{L33} For any prime $p\not =2$ and for $0<j<p$
$${2p\choose j}\equiv -2p{(-1)^j\over j}+4p^2{(-1)^j\over j}H(1;j-1)\pmod{p^3}$$
and
$${2p\choose p}\equiv 2-{4\over 3}p^3 B_{p-3} \pmod{p^4}.$$
\end{lem}
\begin{proof} It suffices to expand the binomial coefficient in this way
$${2p\choose j}=-2p{(-1)^j\over j}\prod_{k=1}^{j-1}\left(1-{2p\over k}\right)
={(-1)^j\over j}\sum_{k=1}^{j-1}(-2p)^k\,H(\{1\}^{k-1};j-1).
$$
and apply (i).
\end{proof}

\begin{proof}[{\sl Proof of Theorem 1.1.}] Letting $n=p$ in the identity given by Corollary \ref{C32} we obtain
$$4^p\sum_{k=1}^{p} {(-1)^k\over k}{-{1\over 2} \choose k}=
4\sum_{0\leq j<d<p}{(-1)^{d}\over p-d} {2p \choose j}+2\sum_{0\leq d<p}{(-1)^{d}\over p-d}{2p \choose d}.$$
that is
$$4^{p-1}\sum_{k=1}^{p-1} {(-1)^k\over k}{-{1\over 2} \choose k}={2-{2p \choose p}\over 4p}
-\sum_{0<d<p}{(-1)^{d}\over d}
+\sum_{0<j<d<p}{(-1)^{d}\over p-d} {2p \choose j}+{1\over 2}\sum_{0<d<p}{(-1)^{d}\over p-d}{2p \choose d}.$$
Now we consider each term of the r.h.s. separately. By Lemma \ref{L33}
$${2-{2p \choose p}\over 4p}\equiv {1\over 3}\,p^2 B_{p-3}\pmod{p^3}.$$
By (ii)
$$\sum_{0<d<p}{(-1)^{d}\over d}=H(-1;p-1)=-2q_p(2)+pq_p(2)^2-{2\over 3}\,p^2 q_p(2)^3-{1\over 4}\,p^2 B_{p-3}\pmod{p^3}.$$
Since for $0<d<p$
$${1\over p-d}=-{1\over d(1-{p\over d})}\equiv - {1\over d}-{p\over d^2} \pmod{p^2}$$
then by Lemma \ref{L33}, (i), and (iii) we have that
\begin{eqnarray*}
\sum_{0<d<p}{(-1)^{d}\over p-d}{2p \choose d}&\equiv&
\sum_{0<d<p}\left(-{(-1)^{d}\over d}-p{(-1)^{d}\over d^2}\right)
\left(-2p{(-1)^d\over d}+4p^2{(-1)^d\over d}H(1;d-1)\right)\\
&\equiv& 2p H(2;p-1)+2p^2 H(3;p-1)-4p^2 H(1,2;p-1)\\
&\equiv&-{8\over 3}\, p^2 B_{p-3}.\pmod{p^3}.
\end{eqnarray*}
In a similar way, by Lemma \ref{L33} and Corollaries \ref{C24} and \ref{C25} we get
\begin{eqnarray*}
\sum_{0<j<d<p}{(-1)^{d}\over p-d} {2p \choose j}&\equiv&
\sum_{0<j<d<p}\left(-{(-1)^{d}\over d}-p{(-1)^{d}\over d^2}\right)
\left(-2p{(-1)^j\over j}+4p^2{(-1)^j\over j}H(1;j-1)\right)\\
&\equiv& 2p H(-1,-1;p-1)+2p^2 H(-1,-2;p-1)-4p^2 H(1,-1,-1;p-1)\\
&\equiv& 4pq_p(2)^2+{4\over 3}\,p^2B_{p-3} \pmod{p^3}.
\end{eqnarray*}
Thus
$$4^{p-1}\sum_{k=1}^{p-1}{(-1)^k\over k}{-{1\over 2} \choose k}=
2q_p(2)+3pq_p(2)^2+{2\over 3}\,p^2 q_p(2)^3+{7\over 12}\,p^2 B_{p-3} \pmod{p^3}.$$
Since $4^{p-1}=(q_p(2)p+1)^2=1+2q_p(2)p+q_p(2)^2p^2$ then
$$4^{-(p-1)}=(1+2q_p(2)p+q_p(2)^2p^2)^{-1}\equiv 1-2q_p(2)p+3q_p(2)^2p^2 \pmod{p^3}.$$
Finally
\begin{eqnarray*}
\sum_{k=1}^{p-1} {(-1)^k\over k}{-{1\over 2} \choose k}
&\equiv& \left(1-2q_p(2)p+3q_p(2)^2p^2\right)
\left(2q_p(2)+3pq_p(2)^2+{2\over 3}\,p^2 q_p(2)^3+{7\over 12}\,p^2 B_{p-3}\right)\\
&\equiv& 2q_p(2)-pq_p(2)^2+{2\over 3}p^2q_p(2)^3+{7\over 12}p^2 B_{p-3} \pmod{p^3}.
\end{eqnarray*}
Note that by (ii) the r.h.s. is just $-H(1,(p-1)/2)=-\sum_{k=1}^{(p-1)/2}{1\over k}$.
\end{proof}

\end{document}